\documentclass[10pt,reqno]{amsart} 

\usepackage[utf8]{inputenc}
\usepackage[T1]{fontenc}	
\usepackage[french,english]{babel}	

\usepackage{lmodern}			

\usepackage{graphicx}	

\usepackage[a4paper,plainpages=false,colorlinks,linkcolor=bleuFonce,citecolor=rougeFonce,urlcolor=vertFonce,breaklinks]{hyperref}

\usepackage{placeins}
\usepackage{float} 

\usepackage{tikz}
\usetikzlibrary{patterns}

\usepackage[normalem]{ulem}

\usepackage{color}
\definecolor{vertFonce}	{rgb}{0,0.5,0}
\definecolor{numLignes}	{rgb}{0.17,0.57,0.7}	
\definecolor{gris}		{rgb}{0.5,0.5,0.5}
\definecolor{grisFonce}	{rgb}{0.2,0.2,0.2}
\definecolor{orange}	{rgb}{1,0.65,0.31}		
\definecolor{orangeFonce}{rgb}{1,0.4,0}
\definecolor{bleuFonce}	{rgb}{0,0,0.4}
\definecolor{rougeFonce}{rgb}{0.3,0,0}
\definecolor{rougeWord}	{rgb}{0.5,0,0}
\definecolor{vertClair}	{rgb}{0.8,1,0.8}
\definecolor{rougeClair}{rgb}{1,0.5,0.5}
\definecolor{violet}	{rgb}{0.5,0,0.5}

\usepackage{pict2e}
\usepackage{multido}

\usepackage{amsfonts,amssymb,amsthm,amsmath} 
\usepackage{dsfont}				
\usepackage{mathrsfs}
\usepackage{yfonts}
\usepackage{cancel}
 
\newtheorem{lem}{Lemma}[section]
\newtheorem{theorem}{Theorem}[section]
\newtheorem{cor}{Corollary}[section]
\newtheorem{prop}{Proposition}[section]

\newtheorem{remark}{Remark}[section]

\usepackage{stmaryrd}
\SetSymbolFont{stmry}{bold}{U}{stmry}{m}{n}
\newcommand		{\subsetArrow}	{\mathrel{\ooalign{$\subset$\cr%
\hidewidth\raise-.087ex\hbox{$_\shortrightarrow\mkern-1.5mu$}\cr}}}
\newcommand		{\subsetarrow}	{\mathrel{\ooalign{$\subset$\cr%
\hidewidth\raise-1.45ex\hbox{$\vec{}\mkern6mu$}\cr}}}

\newcommand		{\N}		{\mathbb N}			
\newcommand		{\RR}		{\mathbb R}			
\newcommand		{\R}		{\RR}

\newcommand		{\Rd}		{\R^3}
\newcommand		{\RRd}		{\R^6}
\renewcommand	{\L}		{\mathcal L}		
\newcommand		{\cW}		{\mathcal W}		

\newcommand		\sfX		{\mathsf X}			
\newcommand		\sfm		{\mathsf m}

\newcommand		{\lt}			{\left}				%
\newcommand		{\rt}			{\right}			%
\renewcommand	{\(}			{\lt(}
\renewcommand	{\)}			{\rt)}

\newcommand		{\com}[1]		{\lt[{#1}\rt]}		

\newcommand		{\n}[1]			{\lt\lvert #1 \rt\rvert}
\newcommand		{\nrm}[1]		{\lt\|{#1}\rt\|}
\newcommand		{\bnrm}[1]		{\big\lVert #1\big\rVert}

\newcommand		{\Nrm}[2]		{\nrm{#1}_{#2}}
\newcommand		{\bNrm}[2]		{\bnrm{#1}_{#2}}

\renewcommand		{\d}		{\mathrm{d}}		
\newcommand			{\dd}		{\,\d}				
			
\newcommand			{\dpt}		{\partial_t}

\newcommand			{\dt}		{\frac{\d}{\d t}}	

\newcommand			{\conj}[1]	{\overline{#1}}		

\DeclareMathOperator{\tr}		{Tr}				
\DeclareMathOperator{\diag}		{diag}

\newcommand		{\Tr}[1]		{\tr\!\( #1 \)}		
\newcommand		{\Diag}[1]		{\diag\!\( #1 \)}

\newcommand		{\intd}			{\int_{\Rd}}

\newcommand		{\iintd}		{\iint_{\RRd}}

\newcommand		{\sumj}			{\sum_{j\in J}}

\newcommand		{\ii}			{\mathrm{i}}	
\newcommand		{\jj}			{\mathrm{j}}	
\newcommand		{\init}			{\mathrm{in}}
\newcommand		{\loc}			{\mathrm{loc}}

\newcommand		{\fb}			{\mathfrak b}
\newcommand		{\eps}			{\varepsilon}

\newcommand		{\cC}			{\mathcal{C}}

\usepackage{braket}

\newcommand		{\op}		{\boldsymbol{\rho}}	

\newcommand		{\opm}		{\sfm}	

\newcommand		{\opmu}		{\boldsymbol{\mu}}	

\newcommand		{\opp}		{\boldsymbol{p}}

\newcommand		{\opA}		{\mathbf{A}}

\newcommand		{\Dh}		{\boldsymbol{\nabla}}	
\newcommand		{\Dhx}		{\Dh_{\!x}}				
\newcommand		{\Dhv}		{\Dh_{\!\xi}}			

\DeclareMathOperator{\adj}	{ad}
\newcommand		{\ad}[1]	{\adj_{#1}}

\title[Global-in-time semiclassical regularity for Hartree--Fock]{Global-in-time semiclassical regularity for the Hartree--Fock equation}

\author{Jacky J. Chong, Laurent Lafleche\\
{\tiny Department of Mathematics, The University of Texas at Austin, Austin, TX 78712, USA}
\\\\ Chiara Saffirio${^*}$
\\
 {\tiny Department of Mathematics and Computer Science, Spiegelgasse 1, 4051 Basel, Switzerland}
}

\subjclass[2010]{35Q55, 35B65, 82C10, 81Q20.}
\keywords{Hartree--Fock equation, Hartree equation, semiclassical, regularity, singular interaction.}

\begin{document}

\begin{abstract}
	For arbitrarily large times $T>0$, we prove the uniform-in-$\hbar$ propagation of semiclassical regularity for the solutions to the Hartree--Fock equation with singular interactions of the form $V(x)=\pm\,|x|^{-a}$ where $a\in(0,\frac12)$. As a byproduct of this result, we extend to arbitrarily long times the derivation of the Hartree--Fock and the Vlasov equations from the many-body dynamics provided in~[J.~Chong, L.~Lafleche, C.~Saffirio: \textit{arXiv:2103.10946} (2021)].
\end{abstract}

\maketitle
\let\thefootnote\relax\footnote{$^*${\it Corresponding author.} chiara.saffirio@unibas.ch}

\bigskip

\pagestyle{plain}


\section{Introduction and Main Result}

	Consider the time-dependent Hartree--Fock equation 
	\begin{equation}\label{eq:HF}
		i\hbar\,\dpt\op = \com{H_{\op},\op}
	\end{equation}
	describing the evolution of a positive self-adjoint trace class operator $\op=\op(t)$ acting on $L^2(\Rd)$. 
	Here $\hbar=\frac{h}{2\pi}$ is the reduced Planck constant, $\com{\cdot,\cdot}$ denotes the commutator $\com{A,B}=AB-BA$, and $H_{\op}$ is the Hamiltonian operator given by
	\begin{equation}\label{eq:HF-hamiltonian}
		H_{\op} = -\frac{\hbar^2\Delta}{2} + V_{\op} - h^3\,\sfX_{\op},
	\end{equation}
	where $V_{\op}$ is the mean-field potential and $\sfX_{\op}$ is the exchange operator. The mean-field potential is defined as the multiplication operator by the function $V_{\op}(x) = (K*\rho)(x)$, where $K:\Rd\to\R$ is the potential associated to some two-body interaction and $\rho(x)$ is the spatial density, defined as the rescaled diagonal of the integral kernel $\op(\cdot,\cdot)$ of the operator $\op$, given by 
	\begin{equation*}
		\rho(x)=\diag(\op)(x):=h^3\op(x,x).
	\end{equation*} 
	 The exchange operator is defined to be the operator with kernel
	\begin{equation*}
		\sfX_{\op}(x,y) = K(x-y)\,\op(x,y).
	\end{equation*}
	In this work, we normalize $\op$ so that
	\begin{equation}\label{eq:fermionic-bounds}
		h^{3}\Tr{\op} = 1,\qquad\text{and}\qquad \cC_\infty:=\Nrm{\op}{\infty},
	\end{equation}
	where $\Nrm{\cdot}{\infty}$ is the operator norm. These quantities are preserved by Equation~\eqref{eq:HF}. Furthermore, we assume the constant in~\eqref{eq:fermionic-bounds} does not depend on $\hbar$.
	With this scaling, we see that $\rho$ satisfies $\intd \rho(x) \dd x = h^3 \Tr{\op} = 1$.
	In the absence of $\sfX_{\op}$, we refer to Equation~\eqref{eq:HF} as the Hartree equation. All the results presented in this work hold for both the Hartree and the Hartree--Fock equation.
	
	In the case when $\hbar$ is fixed, say $\hbar =1$, the well-posedness theories for the Hartree and the Hartree--Fock equations are well known. For the case of the Hartree equation with Coulomb potential, one can find the proof of the global-in-time well-posedness in $L^2$ and the propagation of higher $H^s$ regularity for the wave function in~\cite{castella_$l2$_1997}, which builds on the earlier works \cite{ginibre_class_1980, ginibre_global_1985, illner_global_1994}. The case of density operators in Schatten spaces but with an infinite trace was studied in \cite{lewin_hartree_2015}. In the case of the Hartree--Fock equation, well-posedness in $H^2$ was proved in \cite{bove_existence_1974} for bounded interactions and then in \cite{chadam_time-dependent_1976, bove_hartreefock_1976} for more singular potentials including the case of the Coulomb potential. However, these works do not provide satisfactory estimates when $\hbar$ is small and tending towards zero. Obtaining uniform-in-$\hbar$ estimates is crucial for understanding, of course, the errors in the semiclassical limit $\hbar\to 0$ as in \cite{lions_sur_1993, lafleche_propagation_2019, saffirio_semiclassical_2019, saffirio_hartree_2020, lafleche_global_2021, lafleche_strong_2021}, but also to create adapted numerical schemes~\cite{golse_convergence_2021} and to understand the joint mean-field and semiclassical limit \cite{benedikter_mean-field_2014, benedikter_mean-field_2016-1, porta_mean_2017, chong_many-body_2021}.
	
	In this paper, we are interested in proving the global-in-time propagation of regularity uniformly in the semiclassical parameter $\hbar$ for solutions to the Hartree--Fock equation~\eqref{eq:HF} when the interaction potential $K$ is the inverse power law potential
	\begin{equation*}
		K(x) = \dfrac{\pm 1}{\n{x}^a},\qquad \text{ with }\quad a\in \(0,\tfrac{1}{2}\). 
	\end{equation*}
	In particular, $K\in L^{\fb,\infty}(\Rd)$, where $\fb=\frac{3}{a+1}$, and $L^{\fb,\infty}(\Rd)$ denotes the weak $L^\fb$ space on $\Rd$. Our main motivation is to extend the results of the local-in-time regularity obtained in our previous paper~\cite{chong_many-body_2021} to global-in-time results, leading to the global-in-time mean-field and semiclassical limits for fermions from the $N$ body Schr\"odinger equation to the Hartree--Fock and Vlasov equation.
	
	Before stating our main result, we introduce the function spaces that we will be working with. First, we define the semiclassical phase space Lebesgue norms by
	\begin{equation*}
		\Nrm{\op}{\L^p} := h^{\frac{3}{p}} \Nrm{\op}{p} = h^{\frac{3}{p}}\Tr{\n{\op}^p}^{\frac{1}{p}}
	\end{equation*}
	for $p\in[1,\infty]$, with the obvious modification for $p=\infty$. Here $\Nrm{\op}{p}$ denotes the Schatten norm of order $p$ and $\n{A}=\sqrt{A^*A}$ the absolute value of the operator $A$ with adjoint $A^*$. Let $\opp=-i\hbar\nabla$ be the momentum operator and 
	\begin{equation*}
		\opm_n := 1+\n{\opp}^n,\qquad n\in\N.
	\end{equation*}
	The moment of order $n$ and the weighted semiclassical Lebesgue norms with the operator weight $\opm_n$ are given by 
	\begin{equation*}
		M_n := h^3 \Tr{\n{\opp}^n\op} \qquad\text{and}\qquad \Nrm{\op}{\L^p(\opm_n)} := \Nrm{\op\,\opm_n}{\L^p}.
	\end{equation*}
	In order to consider the quantum analogue of Sobolev norms, we introduce the following operators 
	\begin{equation*}
		\Dhx\op := \com{\nabla,\op}\qquad \text{and}\qquad \Dhv\op := \com{\frac{x}{i\hbar},\op}.
	\end{equation*}
	Then the semiclassical homogeneous Sobolev norms are defined by
	\begin{equation*}
		\Nrm{\op}{\dot{\cW}^{1,p}} := \sum_{k=1}^3 \(\Nrm{\Dh_{x_{k}}\op}{\L^p} + \Nrm{\Dh_{\xi_k}\op}{\L^p}\)
	\end{equation*} 
	with corresponding inhomogeneous Sobolev norms given by $\Nrm{\op}{\cW^{1,p}} := \Nrm{\op}{\L^p} + \Nrm{\op}{\dot{\cW}^{1,p}}$, and the weighted semiclassical Sobolev norms with the operator weight $\opm_n$ by
	\begin{equation*}
		\Nrm{\op}{\cW^{1,p}(\opm_n)} := \Nrm{\op\,\opm_n}{\cW^{1,p}}.
	\end{equation*}
	Our main result states the global-in-time propagation of the regularity in terms of these norms.
	\begin{theorem}\label{thm:prop-reg-grad}
		Let $a\in\(0,\frac{1}{2}\)$, $n\in 2\N$ be an even integer
		and $\op$ be a solution to the Hartree--Fock equation~\eqref{eq:HF} with initial datum $\op^\init\in\L^\infty(\opm_n)$ satisfying~\eqref{eq:fermionic-bounds} and such that
		\begin{equation}
			\op^\init\in\cW^{1,2}(\opm_n)\cap\cW^{1,q}(\opm_n)
		\end{equation}
		for $q\in [2, \infty)$, and with moments of order strictly larger than $\frac{3}{1-a}\(n+a+1\)$ bounded uniformly in $\hbar$. Then
		\begin{equation}\label{eq:regu_op}
			\op\in L^\infty_{\loc}(\R_+,\cW^{1,2}(\opm_n)\cap\cW^{1,q}(\opm_n)\cap\L^\infty(\opm_n))
		\end{equation}
		uniformly in $\hbar\in(0,1)$.
	\end{theorem}
	Notice that Theorem~\ref{thm:prop-reg-grad} extends to arbitrarily long times the local-in-time theory studied in~\cite[Theorem~3.1]{chong_many-body_2021} for $a\in(0,\frac12)$. As a corollary, Theorem~\ref{thm:prop-reg-grad} entails the global-in-time derivation of the Hartree--Fock and the Vlasov equations from the many-body Schr\"{o}dinger equation in the mean-field regime for mixed states, thus extending from local to global-in-time Theorem~3.2 and Theorem~3.3 in~\cite{chong_many-body_2021}. The crucial regularity conditions needed to perform the joint mean-field and semiclassical limit in~\cite{chong_many-body_2021} when $a \in (0,\frac12)$ are indeed of the form~\eqref{eq:regu_op} with $q > \frac{6}{1+2a}$ (take $p=2$ and $q=q_1$ in the beginning of~\cite[Section~10.1]{chong_many-body_2021}).
	
	As in \cite{lafleche_propagation_2019, lafleche_global_2021}, the key ingredient to getting long-time estimates in this work is the usage of quantum moments, used to bound the semiclassical weighted Lebesgue norms. In this paper, we prove the global-in-time propagation of quantum moments for the solution of the Hartree--Fock equation~\eqref{eq:HF} when $a\in(0,\frac45]$ and show that if $a\in(0,\frac12)$ then the global-in-time bound on the moments, combined with a Gr\"{o}nwall-type argument, proves the uniformly-in-$\hbar$ propagation of regularity in weighted semiclassical Sobolev spaces. 
	
	The paper is structured as follows: in Section~\ref{sec:prop-moments}, we recall the result obtained in~\cite{lafleche_propagation_2019} about the global-in-time propagation of quantum moments for the Hartree equation when $a\in(0,\frac45]$ and extend it to solutions of the Hartree--Fock equation~\eqref{eq:HF}, while Section~\ref{sec:proof} is devoted to the proof of Theorem~\ref{thm:prop-reg-grad}.

\section{Propagation of Moments}\label{sec:prop-moments}

	\begin{theorem}\label{thm:propag_moments_HF}
		Let $a\in(0,\frac45]$, $n\in 2\N$ and $\op$ be a solution of the Hartree--Fock equation with initial condition $\op^\init\in\L^1\cap\L^\infty$ with moments of order $n$ bounded uniformly in $\hbar$. Then there exists a continuous function $\Phi_n\in C^0(\R_+)$ independent of $\hbar$ such that for any $t\in\R_+$,
		\begin{equation*}
			M_n(t) \leq \Phi_n(t).
		\end{equation*}
	\end{theorem}
	
	\begin{remark}
		If $a\in(\frac45,2)$, then we can still get a short-time estimate when $a \leq a_n = \frac{2\,n}{n+3}$ (see Remark~\ref{rmk:a_max}).
		In particular, $a_n$ is larger than $1$ as soon as $n\geq 3$. More generally, for any $a\in(0,2)$, the propagation of moment of order $n$ holds for any even $n\geq n_a = \frac{3\,a}{2-a}$.
	\end{remark}
	
	\begin{remark}
		The proof of the theorem can be used to get an explicit function $\Phi_n(t)$. It has a polynomial growth in time when $a < \frac45$, and an exponential growth in time when $a = \frac45$.
	\end{remark}

\subsection{The Hartree Equation}\label{sec:moments_Hartree}
	
	In this section, we recall the main ingredients of the proof of the propagation of moments for the Hartree equation obtained in~\cite[Theorem~3]{lafleche_propagation_2019}. This provides us a guide for the extension to the case of the Hartree--Fock equation addressed in Theorem~\ref{thm:propag_moments_HF}. 

	Recall that for any density operator $\op$ (i.e. any positive trace class operator with trace one), there exists $J\subset\N$, a sequence of functions $\(\psi_j\)_{j\in J}$ orthonormal in $L^2(\Rd)$ and a positive summable sequence $\(\lambda_j\)_{j\in J}$ such that $\op$ can be written as
	\begin{equation*}
		\op = \sumj \lambda_j \ket{\psi_j} \bra{\psi_j}.
	\end{equation*}
	For any even integer $n\in 2\N$, we define the moment density of order $n$ by
	\begin{equation}\label{eq:local_moments}
		\rho_n(x) := h^3\sumj \lambda_j\n{\opp^\frac{n}{2}\psi_j(x)}^2 = \Diag{\opp^\frac{n}{2}\op\cdot\opp^\frac{n}{2}}(x)
	\end{equation}
	so that the moment of order $n$, previously defined, can be rewritten as $M_n = \Nrm{\rho_n}{L^1}$. Notice we also have $M_n = \bNrm{\sqrt{\op}\n{\opp}^{\frac{n}{2} }}{\L^2}^2$. With this notation, Inequality~(38) in~\cite{lafleche_propagation_2019} reads
	\begin{equation}\label{eq:bound_tr_com_V}
		\n{h^3 \Tr{\tfrac{1}{i\hbar}\com{V_{\op},\n{\opp}^n}\op}} \leq C\, M_n^\frac{1}{2} \sup_{\substack{(j,k,l)\in\N^3\\j+k+l=n/2-1}} \Nrm{\rho_{2j}}{L^\alpha}^\frac{1}{2} \Nrm{\rho_{2k}}{L^\beta}^\frac{1}{2} \Nrm{\rho_{2l}}{L^\gamma}^\frac{1}{2},
	\end{equation}
	where
	\begin{equation*}
		\frac{1}{\alpha'} + \frac{1}{\beta'} + \frac{1}{\gamma'} = \frac{1}{\fb}
	\end{equation*}
	and $\alpha',\,\beta',\,\gamma'$ are the H\"{o}lder conjugates of $\alpha,\,\beta,\,\gamma$ respectively. The semiclassical kinetic inequality~\cite[Theorem~6]{lafleche_propagation_2019}, which is a generalization of the Lieb--Thirring inequality for the $n$th order moment density, tells us that for any $(k,n)\in(2\N)^2$ verifying $k\in[0,n]$ we have that 
	\begin{align}\label{eq:interpolation}
		\Nrm{\rho_k}{L^p} & \leq C \Nrm{\op}{\L^\infty}^\frac{1}{p'} M_n^\frac{1}{p}\qquad \text{ with } \qquad p=p_{n, k}:= \frac{3+n}{3+k}.
	\end{align}
	Combining this inequality with Inequality~\eqref{eq:bound_tr_com_V} implies the existence of some positive constants $\Theta$, $\Theta_0$, and $\Theta_2$ such that the following estimate (see \cite[Inequality~(44)]{lafleche_propagation_2019})
	\begin{equation*}
		\n{h^3 \Tr{\tfrac{1}{i\hbar}\com{V_{\op},\n{\opp}^n}\op}} \leq C \Nrm{\op}{\L^\infty}^{\Theta_2} M_{n-2}^{\Theta_0}\, M_{n}^{\Theta}
	\end{equation*}
	holds, where $\Theta \leq 1$ when $a\leq \frac45$. This leads to the boundedness of moments by a Gr\"onwall-type argument, together with the uniform boundedness of $M_2$ due to the conservation of energy and an induction on $n$. Moreover, when $a < \frac45$, then $\Theta < 1$, thus proving a polynomial growth in time in this latter case.

\bigskip
\subsection{The Hartree--Fock Equation}
	
	We now consider the Hartree--Fock equation. By estimating the exchange (operator) term with a similar strategy, we get the analogue of~\cite[Theorem~3]{lafleche_propagation_2019}. The analogue of Inequality~\eqref{eq:bound_tr_com_V} for the exchange term is given by the following lemma.
	
	\begin{lem}\label{lem:bound_tr_com_X}
		Denote the components of $\opp$ by $\opp_{x_l}$ or simply $\opp_l$ for $l\in\set{1,2, 3}$. Suppose $n\in\N$ and define $Q\subset [\frac{1}{2},\infty]^4$ the set of $(q_1,q_2,q_3,q_4)$ verifying $\frac{1}{q_1}+\frac{1}{q_2} \in (0,2)$, $\frac{1}{q_3}+\frac{1}{q_4} \in (0,2)$ and
		\begin{equation}\label{eq:HLS-Holder_coefs}
			\frac{1}{q_1'} + \frac{1}{q_2'} + \frac{1}{q_3'} + \frac{1}{q_4'} = \frac{2}{\fb}.
		\end{equation}
		Then there exists a constant $C$ independent of $\hbar$ such that
		\begin{equation}\label{eq:bound_tr_com_X}
			\n{\frac{h^3}{\hbar} \Tr{\com{h^3\sfX_{\op},\opp_l^n}\op}} \leq C \sup_{k_1,\dots k_4} \inf_Q \Nrm{\rho_{k_1}}{L^{q_1}}^\frac{1}{2} \Nrm{\rho_{k_2}}{L^{q_2}}^\frac{1}{2} \Nrm{\rho_{k_3}}{L^{q_3}}^\frac{1}{2} \Nrm{\rho_{k_4}}{L^{q_4}}^\frac{1}{2},
		\end{equation}
		where the supremum is taken over all the integers $(k_1,\dots, k_4) \in \N^4$ such that $k_1+\dots+k_4 = 2\(n-1\)$.
	\end{lem}
	
	\begin{proof}
		Let us begin by making the observation that $[\sfX_{\op}, \opp_l^n]\op$ is trace class.
		By the Leibniz formula, the kernel of $\opp_l^n\sfX_{\op}\op$ is given by
		\begin{multline*}
			\opp_l^n\sfX_{\op} \op(x, y) = \sum^{n}_{k=0} \binom{n}{k} \intd \opp_{x_l}^{n-k} K(x-z)\,\opp_{x_l}^{k}\op(x, z)\,\op(z, y)\dd z
			\\
			= \sum^{n}_{k=0}\sum_{\ell=0}^{n-k} \binom{n}{k}\binom{n-k}{\ell}(-1)^{n-k}\intd K(x-z)\(\opp_{x_l}^{k}\opp_{z_l}^{\ell}\op(x, z)\)\opp_{z_l}^{n-k-\ell}\op(z, y)\dd z
			\\
			=: \sum^{n}_{k=0}\sum_{\ell=0}^{n-k} \binom{n}{k}\binom{n-k}{\ell}(-1)^{n-k+1}\opA_{n, k, l}(x, y).
		\end{multline*}
		It is clear that each $\opA_{n,k,l}$ is a product of two Hilbert--Schmidt operators, and so is trace class, provided $\op$ is sufficiently regular. Since $\sfX_{\op}\opp_l^n\op = \opA_{n,0,0}$ is also trace class, it follows that $[\sfX_{\op}, \opp_l^n]\op$ is indeed trace class. 

		Let $I_n := \Tr{\com{\sfX_{\op},\opp_l^n}\op}$. Then, by \cite[VI.7 Theorem 17]{gaal_linear_1973}, we may express the trace of $\com{\sfX_{\op},\opp_l^n}\op$ in terms of its kernel as follows
		\begin{multline*}
			I_n = \iintd K(x-y)\,\op(y,x)\, \opp_{x_l}^n( \op)(x,y) - \opp_{x_l}^n\!\(K(x-y)\op(x,y)\) \op(y,x)\dd x\dd y.
		\end{multline*}
		Therefore, using the Leibniz formula and then diagonalizing the self-adjoint compact operator $\op$ yields
		\begin{multline*}
			I_n = -\sum_{k=0}^{n-1} \binom{n}{k}\iintd (\opp_{x_l}^{n-k}K)(x-y)\(\opp_{x_l}^k\op(x,y)\) \op(y,x)\dd x\dd y
			\\
			= -\sum_{(\ii,j)\in J^2}\sum^{n-1}_{k=0} \binom{n}{k} \lambda_{\ii}\, \lambda_j \intd (\opp_{l}^{n-k}K)*\!\(\conj{\psi_\ii}\,\psi_j\) \opp_{l}^k\!\(\psi_\ii\)\conj{\psi_j},
		\end{multline*}
		and so
		\begin{align*}
			\frac{I_n}{i\hbar} &= \sum_{(\ii,j)\in J^2} \sum^{n-1}_{k=0}\binom{n}{k} \lambda_{\ii}\, \lambda_j \intd (\partial_{l}K)*\opp_l^{n-1-k}\!\(\conj{\psi_\ii}\,\psi_j\) \opp_{l}^k\!\(\psi_\ii\)\conj{\psi_j}.
		\end{align*}
		Let $I_{n,k}$ denote the integral in the above formula. We need to balance the powers of $\opp_l$ in $I_{n,k}$. For simplicity of notation, we assume $n=2\,\widetilde{n}\in 2\,\N$. The case when $n$ is odd is similar. Then, if $k\leq\widetilde{n}$, we have
		\begin{align*}
			I_{n,k} &= \intd (\partial_{l}K)*\opp_l^{\tilde{n}-1}\!\(\conj{\psi_\ii}\,\psi_j\) \opp_{l}^{\tilde{n}-k}\!\(\opp_{l}^k\!\(\psi_\ii\)\conj{\psi_j}\)
			= \sum_{\ell=0}^{\tilde{n}-k}\sum_{\tilde{\ell}=0}^{\tilde{n}-1}\binom{\widetilde{n}-k}{\ell} \binom{\widetilde{n}-1}{\widetilde{\ell}} I_{n,k,\ell,\tilde{\ell}},
		\end{align*}
		where
		\begin{align*}
			I_{n,k,\ell,\tilde{\ell}} &= \intd (\partial_{l}K)*\!\(\opp_l^{\tilde{\ell}}(\conj{\psi_\ii})\,\opp_l^{\tilde{n}-1-\tilde{\ell}}(\psi_j)\) \opp_{l}^{\tilde{n}-\ell}(\psi_\ii)\, \opp_{l}^{\ell}(\conj{\psi_j}).
		\end{align*}
		Now notice that by the Cauchy-Schwarz inequality for sums and Definition~\eqref{eq:local_moments} for the moment density, we have that
		\begin{align*}
			h^{2d}\sum_{(\ii,j)\in J^2} \lambda_{\ii}\, \lambda_j \, I_{n,k,\ell,\tilde{\ell}} &\leq \intd \n{\partial_{l}K}*\!\(\rho_{2\tilde{\ell}}^\frac{1}{2}\,\rho_{2(\tilde{n}-1-\tilde{\ell})}^\frac{1}{2}\) \rho_{2(\tilde{n}-\ell)}^\frac{1}{2}\, \rho_{2\ell}^\frac{1}{2}
			\\
			&\leq C_K \Nrm{\rho_{2\tilde{\ell}}}{L^{q_1}}^\frac{1}{2} \Nrm{\rho_{2(\tilde{n}-1-\tilde{\ell})}}{L^{q_2}}^\frac{1}{2} \Nrm{\rho_{2(\tilde{n}-\ell)}}{L^{q_3}}^\frac{1}{2} \Nrm{\rho_{2\ell}}{L^{q_4}}^\frac{1}{2},
		\end{align*}
		where the last inequality follows from the Hardy--Littlewood--Sobolev inequality and the H\"older inequality. Similarly, when $k>\widetilde{n}$, we have that
		\begin{align*}
			I_{n,k} :=&\, \intd \opp_{l}^{k-\tilde{n}}\!\((\partial_{l}K)*\opp_l^{n-k-1}\!\(\conj{\psi_\ii}\,\psi_j\) \conj{\psi_j}\) \opp_{l}^{\tilde{n}}\!\(\psi_\ii\)
			\\
			=&\, \sum_{\ell=0}^{k-\tilde{n}} \sum_{\tilde{\ell}=0}^{\tilde{n}-1-\ell} \binom{k-\widetilde{n}}{\ell} \binom{\widetilde{n}-1-\ell}{\widetilde{\ell}} I_{n,k,\ell,\tilde{\ell}}
		\end{align*}
		where
		\begin{align*}
			I_{n,k,\ell,\tilde{\ell}} &= \intd (\partial_{l}K)*\!\(\opp_l^{\tilde{\ell}}(\conj{\psi_\ii})\,\opp_l^{\tilde{n}-1-\ell-\tilde{\ell}}(\psi_j)\) \opp_{l}^{\ell}(\conj{\psi_j}) \, \opp_{l}^{\tilde{n}}(\psi_\ii).
		\end{align*}
		Mimicking the estimates we obtained above for the case $k\leq \widetilde{n}$ completes the proof of the lemma. 
	\end{proof}
	
	Now combining the semiclassical kinetic interpolation inequality~\eqref{eq:interpolation} for $n=2$ and another $n\in\N$, we get the following inequalities (See \cite[Propostion~3.2]{lafleche_global_2021}).
	\begin{lem}
		Let $0\leq k\leq n$ and $p_{n,k}' = \(\frac{n}{k}\)'\(1+\frac{3}{n}\) = \frac{3+n}{n-k}$. Then for any $p\in[1,p_{n,k}]$, if $k\geq 2$ or if $p\geq p_{2, 0}=\frac53$, it holds
		\begin{equation*}
			\Nrm{\rho_k}{L^p} \leq C_{n,k}^{\frac{1}{p'}}\,\Nrm{\op}{\L^\infty}^\frac{1}{p'} M_2^{\theta_2} \, M_n^{\theta_n}
		\end{equation*}
		with $\theta_2 = \frac{n-k}{n-2} - \frac{3+n}{n-2}\frac{1}{p'}$ and $\theta_n = \frac{k-2}{n-2} + \frac{5}{n-2}\frac{1}{p'}$. If $k=0$ and $p\leq p_{2, 0}$ then
		\begin{equation*}
			\Nrm{\rho}{L^p} \leq C^{\frac{1}{p'}} \Nrm{\op}{\L^\infty}^\frac{1}{p'} M_0^{1-\frac{5}{2\,p'}} \, M_2^\frac{3}{2\,p'}.
		\end{equation*}
	\end{lem}
	
	These two inequalities can be merged into a single inequality in terms of the non-homogeneous moments $1+M_n = h^3 \Tr{\op\,\opm_n}$.
	
	\begin{cor}
		For any $n\geq 2$, $k\in[0,n]$ and $p\in[1,p_{n,k}]$, there exists a constant $C>0$ such that for any compact operator $\op$ we have the following estimate
		\begin{equation}\label{eq:M2_Mn_interpolation}
			\Nrm{\rho_k}{L^p} \leq C \Nrm{\op}{\L^\infty}^\frac{1}{p'} \(1+M_2\)^{\theta_2} \(1+M_n\)^{\theta_n},
		\end{equation}
		with $\theta_n = \frac{k-2}{n-2} + \frac{5}{n-2}\frac{1}{p'}$.
	\end{cor}
	
	\begin{remark}\label{rmk:a_max}
		If $(q_j)_{j=1,\ldots, 4}$ verify $q_j \in [1,p_{n,k_j}]$ and $\sum_{j=1}^4 k_j = 2\(n-1\)$, then $q'_j \in [p'_{n,k_j},\infty]$ so
		\begin{equation*}
			0\leq \sum_{j=1}^4 \frac{1}{q'_j} \leq \sum_{j=1}^4 \frac{n-k_j}{n+3} = \frac{2\(n+1\)}{n+3}.
		\end{equation*}
		Hence, we can find such a family verifying~\eqref{eq:HLS-Holder_coefs} as soon as $\frac{2\(n+1\)}{n+3} \geq \frac{2}{\fb}$, or equivalently as soon as $a \leq \frac{2\,n}{n+3} =: a_n$.
	\end{remark}
	
	We can now complete the proof of Theorem~\ref{thm:propag_moments_HF} using a similar strategy as to the one explained in Section~\ref{sec:moments_Hartree} for the Hartree equation. Instead of proceeding by induction and bounding the time derivative of $M_n$ by a product involving $M_n$ and $M_{n-2}$, we directly estimate it by a product involving $M_n$ and $M_2$. This method allows us to improve slightly the result of~\cite[Theorem~3]{lafleche_propagation_2019}, even in the case of the Hartree equation, as it allows us to propagate moments of high order locally in time for any $a\in(0,2)$ while \cite[Theorem~3]{lafleche_propagation_2019} only covers the case $a\in(0,\frac87)$ in dimension~$3$.
	
	\begin{proof}[Proof of Theorem~\ref{thm:propag_moments_HF}]
		Taking the time derivative of moments and using the cyclicity of the trace yields
		\begin{equation*}
			i\hbar\,\dt\,\Tr{\op\, \opp_{l}^n} = - \Tr{\com{V_{\op},\opp^n_l}\op} - \Tr{\com{\sfX_{\op},\opp^n_l}\op}.
		\end{equation*}
		Then, by Inequality~\eqref{eq:bound_tr_com_V}, which also holds for $\opp_l^n$ in place of $|\opp|^n$, and Inequality~\eqref{eq:bound_tr_com_X}, we deduce that
		\begin{equation*}
			\n{h^3\dt\Tr{\op\, \opp_{l}^n}} \leq C_{n,a} \sup_{k_1+\dots+k_4 = 2\(n-1\)} \inf_Q \Nrm{\rho_{k_1}}{L^{q_1}}^\frac{1}{2} \Nrm{\rho_{k_2}}{L^{q_2}}^\frac{1}{2} \Nrm{\rho_{k_3}}{L^{q_3}}^\frac{1}{2} \Nrm{\rho_{k_4}}{L^{q_4}}^\frac{1}{2}
		\end{equation*}
		with the notations of Lemma~\ref{lem:bound_tr_com_X}. Using the interpolation formula~\eqref{eq:M2_Mn_interpolation} for each of the terms in the right-hand side of the above inequality, we get
		\begin{equation*}
			\n{\frac{h^3}{i\hbar}\Tr{\com{h^3\sfX_{\op},\opp_l^n}\op}} \leq C_{n,a} \Nrm{\op}{\L^\infty}^\frac{1}{\fb} \(1+M_2\)^{\Theta_2} \(1+M_n\)^{\Theta_n}
		\end{equation*}
		with $\Theta_2 = \frac{n+1}{n-2} - \frac{n+3}{n-2}\frac{1}{\fb}$ and \begin{equation*}
			\Theta_n = \frac{2n-10}{2\(n-2\)} + \frac{5}{n-2}\frac{1}{\fb} = 1 + \frac{1}{n-2} \(\frac{5}{\fb} - 3\).
		\end{equation*}
		Moreover, by \cite[Lemma 6.3]{chong_many-body_2021}, $M_n$ can be further bounded by 
		\begin{align*}
			C^{-1} \sum_{l=1}^3\Tr{\op\, (1+\opp_l^n)} \le M_n \le C \sum_{l=1}^3\Tr{\op\, (1+\opp_l^n)} 
		\end{align*}
		for some $C>0$. In particular
		\begin{equation*}
			\Theta_n \leq 1 \iff \fb \geq \tfrac{5}{3} \iff a \leq \tfrac{4}{5}
		\end{equation*}
		which yields the result by Gr\"onwall's Lemma.
	\end{proof}

\section{Proof of Theorem~\ref{thm:prop-reg-grad}}\label{sec:proof}

	Notice that bounds on the moments imply bounds on the semiclassical Schatten norms for any $p\in[2,\infty)$. More precisely, we have the following proposition.
	\begin{prop}\label{prop:p-schatten_moments}
	Let $\op$ be a positive trace class operator. Then, for any $p\in[2,\infty)$, there exists a constant $C>0$ such that
	\begin{equation*}
		\Nrm{\op\, \opm_n}{\L^p} \leq C \Nrm{\op}{\L^\infty}^\frac{1}{p'} \(1+M_{np}\)^\frac{1}{p}.
	\end{equation*}
	\end{prop}
	
	\begin{proof}
		Using the fact that $\op$ is a bounded operator, we get that
		\begin{equation*}
			\Nrm{\op\,\opm_n}{\L^p} \leq \Nrm{\op}{\L^\infty}^\frac{1}{p'} \Nrm{\op^\frac{1}{p}\,\opm_n}{\L^p} \leq \Nrm{\op}{\L^\infty}^\frac{1}{p'} \Nrm{\sqrt{\op}\,\opm_n^\frac{p}{2}}{\L^2}^\frac{2}{p} = \Nrm{\op}{\L^\infty}^\frac{1}{p'} h^3\Tr{\op\,\opm_n^p}^\frac{1}{p},
		\end{equation*}
		where the second inequality follows from the Araki--Lieb--Thirring inequality.
	\end{proof}
	
	For $p=\infty$, we control $\Nrm{\op\,\opm_n}{\L^\infty}$ by means of a Gr\"onwall argument. We will need the following commutator estimates in the spirit of \cite{lafleche_propagation_2019, chong_many-body_2021}, but improved to lead to large time estimates.
	\begin{prop}[Weighted commutator estimate]\label{prop:weighted_com_est}
		Let $a \in (0,\frac12)$, $n\in \N$ and $q\in [1,\infty]$. Then for any $\eps\in(0,1)$ and $n_1 > n+a+1$, there exists a constant $C>0$ such that for every compact self-adjoint operators $\op$ and $\opmu$
		\begin{align}\label{eq:weighted_com_est_E}
			\frac{1}{\hbar}\Nrm{\com{\nabla V_{\op},\opp_\jj^{n}}\opmu}{\L^q} &\leq C \Nrm{\op \,\opm_{n_1}}{\L^{r\pm\eps}} \Nrm{\opmu \,\opm_{n}}{\L^q},
			\\\label{eq:weighted_com_est_V}
			\frac{1}{\hbar}\Nrm{\com{V_{\op},\opp_\jj^{n}}\opmu}{\L^q} &\leq C \(\Nrm{\op \,\opm_{n_1}}{\L^{r\pm\eps}} + M_0\) \Nrm{\opmu \,\opm_{n}}{\L^q},
		\end{align}
		where $r = \tfrac{3}{1-a}$ and $\Nrm{\cdot}{\L^{r\pm\eps}}$ stands for the norm
		\begin{equation*}
			\Nrm{\cdot}{\L^{r\pm\eps}} := \Nrm{\cdot}{\L^{r+\eps}\cap\L^{r-\eps}} = \Nrm{\cdot}{\L^{r+\eps}}+\Nrm{\cdot}{\L^{r-\eps}}.
		\end{equation*}
	\end{prop}

	\begin{proof}[Proof of Proposition~\ref{prop:weighted_com_est}]
		We first proceed as in \cite{lafleche_propagation_2019, chong_many-body_2021} and write
		\begin{equation*}
			\frac{1}{\hbar}\Nrm{\com{\nabla V_{\op},\opp_\jj^{n}}\opmu\,}{\L^q} \leq \sum_{\ell=0}^{n-1} C_\ell \Nrm{g_\ell\, \,\opp_\jj^{n-1-\ell}\opmu\,}{\L^q} \leq \sum_{\ell=0}^{n-1} C_\ell \Nrm{g_\ell}{L^\infty}\Nrm{\opmu\,\opm_{n}}{\L^q},
		\end{equation*}
		where $g_\ell$ is the function defined by $g_\ell(x) = (\opp_\jj^\ell(\partial_{\jj} \nabla V_{\op}))(x)$ and $C_\ell = \sum_{k=\ell}^{n-1} \binom{k}{\ell}$. Noticing that
		\begin{align*}
			g_\ell = -\partial_{\jj}\nabla K * (\opp_\jj^\ell\rho) = -\partial_{\jj}\nabla K * \Diag{\ad{\opp_\jj}^\ell({\op})}
		\end{align*}
		where $\ad{A}(X):=[A, X]$ and using the fact that $\partial_{\jj}\nabla K \in L^{\frac{3}{a+2}+\epsilon} + L^{\frac{3}{a+2}-\epsilon}$ for any $\epsilon>0$ sufficiently small, then Young's inequality yields
		\begin{align*}
			\Nrm{g_\ell}{L^\infty} \leq C_K \Nrm{\Diag{\ad{\opp_\jj}^\ell\!(\op)}}{L^{r\pm\eps}}
		\end{align*}
		for any $\eps\in(0,1)$. By \cite[Proposition~6.4 and Lemma~6.5]{chong_many-body_2021}, it implies that
		\begin{align*}
			\Nrm{g_\ell}{L^\infty} \leq C_K \Nrm{\ad{\opp_\jj}^\ell\!\({\op}\) \opm_{n_0}}{\L^{r\pm\eps}} \leq 2^{\ell+1}\, C_K \Nrm{{\op}\,\opm_{n_0+n-1}}{\L^{r\pm\eps}}
		\end{align*}
		with $n_0 > 3/r'$. This proves Inequality~\eqref{eq:weighted_com_est_E}. When $\nabla V_{\op}$ is replaced by $V_{\op}$, then we just replace $\nabla K$ by $K$ and so we need to find a $L^\infty$ bound for the function
		\begin{align*}
			g_\ell = -\partial_{\jj} K * (\opp_\jj^\ell\rho).
		\end{align*}
		When $\ell>0$, we just write $g_\ell = i\hbar\,\partial_{\jj}^2 K * (\opp_\jj^{\ell-1}\rho)$ and then use the same estimates as for $\nabla V_{\op}$. If $\ell=0$, then $g_\ell = -\partial_{\jj} K * \rho$ is bounded using Young's inequality by
		\begin{equation*}
			\Nrm{g_\ell}{L^\infty} \leq \Nrm{\rho}{L^{\fb'\pm\eps}} \leq \Nrm{\rho}{L^r} + \Nrm{\rho}{L^1}
		\end{equation*}
		since $\fb' \leq r$. By \cite[Proposition~6.4]{chong_many-body_2021}, $\Nrm{\rho}{L^r} \leq \Nrm{\op \,\opm_{n_0}}{\L^{r}}$ with $n_0>3/r' = a+2$.
	\end{proof}
	
	\begin{prop}\label{prop:gronwall-infty}
	Let $a\in\(0,\frac{1}{2}\)$ and $\op$ be a solution to the Hartree--Fock equation~\eqref{eq:HF}. Assume $\op^\init\in\L^\infty(\opm_n)$ for some even $n\in 2\N$ with moments of order $\(n+a+1\)\big(\frac{3}{1-a}+\eps\big)$ bounded uniformly in $\hbar$ for some $\eps\in\,\big(0,1\big)$. Then
	\begin{equation*}
		\op\in L^\infty_{\loc}(\R_+,\L^\infty(\opm_n)).
	\end{equation*}
	\end{prop}
	\begin{proof}
		By \cite[Lemma 6.2]{chong_many-body_2021} with $q\rightarrow \infty$, we have
		\begin{equation}\label{eq:infty-op-m}
		\dt\Nrm{\op\,\opm_n}{\L^\infty}\leq \dfrac{1}{\hbar}\Nrm{\com{V_{\op},\opm_n}\op}{\L^\infty} + \dfrac{1}{\hbar}\Nrm{\com{h^3\sfX_{\op},\opm_n}\op}{\L^\infty}.
		\end{equation}
		Next, we use Proposition~\ref{prop:weighted_com_est} to bound the first term on the right-hand side of Inequality~\eqref{eq:infty-op-m}, leading to\footnote{We use the standard notation $x\lesssim y$ to denote $x\leq Cy$ for some constant $C>0$ which is independent of $x$, $y$ and $\hbar$.}
		\begin{equation*}
			\dfrac{1}{\hbar}\Nrm{\com{V_{\op},\opm_n}\op}{\L^\infty}\lesssim \(1+\Nrm{\op\,\opm_{n_1}}{\L^{r\pm\eps}}\) \Nrm{\op\,\opm_n}{\L^\infty},
		\end{equation*}
		where $r=\frac{3}{1-a}$. Since $\eps\leq 1$ and $r\geq 3$, we see that $r-\eps\geq 2$ which allows us to apply Proposition~\ref{prop:p-schatten_moments} to get
		\begin{equation}\label{eq:direct-term-infty}
			\dfrac{1}{\hbar}\Nrm{\com{V_{\op},\opm_n}\op}{\L^\infty} \lesssim \(1+\Nrm{\op}{\L^\infty} + M_{n_1\(r+\eps\)}\) \Nrm{\op\,\opm_n}{\L^\infty} .
		\end{equation}
		As for the second term on the right-hand side of Inequality~\eqref{eq:infty-op-m}, we first observe that by \cite[Lemma~6.4]{chong_many-body_2021} 
		\begin{equation*}
			\hbar^{\frac32-a} \Nrm{\Dhx \op\, \opm_n}{\L^p} = \hbar^{\frac12-a} \Nrm{\com{\opp,\op} \opm_n}{\L^p} \leq 2\,\hbar^{\frac12-a} \Nrm{\op \,\opm_{n+1}}{\L^p},
		\end{equation*}
		and then apply \cite[Proposition~6.8]{chong_many-body_2021} with $\opm_n = 1+|{\opp}|^n$ to get
		\begin{equation*}
			\dfrac{1}{\hbar}\Nrm{\com{h^3\sfX_{\op},\opm_n} \op}{\L^\infty} \lesssim \hbar^{\frac{1}{2}-a} \Nrm{\op\,\opm_{n+1}}{\L^2} \Nrm{\op\,\opm_n}{\L^\infty}.
		\end{equation*}
		Proposition~\ref{prop:p-schatten_moments} allows us to bound $\Nrm{\op\,\opm_{n+1}}{\L^2}$ in terms of $M_{2\(n+1\)}$ and obtain the following inequality for the left-hand side of Equation~\eqref{eq:infty-op-m}
		\begin{equation*}
			\dt\Nrm{\op\,\opm_n}{\L^\infty} \lesssim \(1 + \Nrm{\op}{\L^\infty} + M_{n_1\(r+\eps\)} + M_{2\(n+1\)}\) \Nrm{\op\,\opm_n}{\L^\infty}
		\end{equation*}
		that gives a bound on $\Nrm{\op\,\opm_n}{\L^\infty}$ by means of Gr\"{o}nwall's Lemma and the control on moments established in Theorem~\ref{thm:propag_moments_HF}.
	\end{proof}
	
	Summarizing, Theorem~\ref{thm:propag_moments_HF} and the two propositions~\ref{prop:p-schatten_moments} and~\ref{prop:gronwall-infty} imply the following result.
	\begin{cor}\label{cor:propagation_semiclassical_norms}
		Let $a\in\(0,\frac{1}{2}\)$ and $\op$ be a solution to the Hartree--Fock equation~\eqref{eq:HF}. For $n\in 2\N$ and let $\op^\init\in\L^\infty(\opm_n)$ with moment of order $\(n+a+1\)\(\frac{3}{1-a}+\eps\)$ bounded uniformly in $\hbar$ for some $\eps\in\(0,1\)$. Then for any $q\in[2,\infty]$, we have that
		\begin{equation*}
			\op\in L^\infty_{\loc}(\R_+,\L^q(\opm_n)).
		\end{equation*}
	\end{cor}
	
	\begin{proof}[Proof of Theorem~\ref{thm:prop-reg-grad}]
		To propagate quantum Sobolev norms, we proceed as in \cite[Section~6]{chong_many-body_2021} and consider the inequalities
		\begin{multline*}
			\dt\Nrm{\Dhx\op\,\opm_n}{\L^q} \leq \dfrac{1}{\hbar}\(\Nrm{\com{V_{\op},\opm_n}\Dhx\op}{\L^q} + \Nrm{\com{\nabla V_{\op},\opm_n} \op}{\L^q}\) + \dfrac{1}{\hbar}\Nrm{\com{\nabla V_{\op},\op\,\opm_n} }{\L^q}
			\\
			+ \dfrac{1}{\hbar}\Nrm{\com{h^3\sfX_{\op},\opm_n}\Dhx\op}{\L^q} + \dfrac{1}{\hbar}\Nrm{\com{h^3\sfX_{\Dhx\op},\op}\opm_n}{\L^q}
		\end{multline*}
		and
		\begin{multline*}
			\dt\Nrm{\Dhv\op\,\opm_n}{\L^q} \leq \dfrac{1}{\hbar} \Nrm{\com{V_{\op},\opm_n}\Dhv\op}{\L^q} + \Nrm{\Dhx\op\,\opm_n}{\L^q}
			\\
			+ \dfrac{1}{\hbar} \Nrm{\com{h^3\sfX_{\op},\opm_n}\Dhv\op}{\L^q} + \dfrac{1}{\hbar} \Nrm{\com{h^3\sfX_{\Dhv\op},\op}\opm_n}{\L^q}
		\end{multline*}
		and we define $N_q = N_q(t) := \Nrm{\op}{\cW^{1,q}(\opm_n)}$. Now, when $q \leq \frac{6}{1+2a}$, \cite[Proposition~6.5]{chong_many-body_2021} yields an estimate of the form
		\begin{equation}\label{eq:com_est_0}
			\dfrac{1}{\hbar}\Nrm{\com{\nabla V_{\op},\op\,\opm_n} }{\L^q} \leq \Nrm{\rho}{L^{r\pm\eps}\cap L^2} \Nrm{\Dhx\!\(\op\,\opm_n\)}{\L^q}
		\end{equation}
		with $r=\frac{3}{1-a}$. Then, using the above inequality, Proposition~\ref{prop:weighted_com_est} and, similarly as in the proof of the previous proposition, \cite[Proposition~6.8 and Proposition~6.9]{chong_many-body_2021} to bound the terms involving the exchange term, we obtain for any $q \leq \frac{6}{1+2a}$
		\begin{multline*}
			\dt N_q \lesssim \(1+\Nrm{\op\,\opm_{n_1}}{\L^{r\pm\eps}}\) N_q + \Nrm{\rho}{L^{r\pm\eps}\cap L^2} N_q
			\\
			+ h^{\frac{1}{2}-a} \Nrm{\op\,\opm_{n+1}}{\L^2} N_q + h^{3\(\frac{1}{q}+\frac{1}{2}-\frac{1}{\fb}\)} \Nrm{\op\,\opm_n}{\L^\infty} N_2,
		\end{multline*}
		with $n_1 > n+a+1$. Notice that for $a\in(0,\frac12)$ the quantity $\frac{1}{q}+\frac{1}{2}-\frac{1}{\fb}$ is positive for every $q\in[2,\infty)$. Hence, for $h\in(0,1)$, we have that 
		\begin{align*}
			\dt\(N_2+N_q\) &\lesssim \(1+\Nrm{\op\,\opm_{n_1}}{\L^{r\pm\eps}}+\Nrm{\op\,\opm_{n}}{\L^\infty} + \Nrm{\op\,\opm_{n+1}}{\L^2}\) \(N_2+N_q\)
		\end{align*}
		which proves Theorem~\ref{thm:prop-reg-grad} for $q \leq \frac{6}{1+2a}$ by Corollary~\ref{cor:propagation_semiclassical_norms} and Gr\"onwall's Lemma. The limitation on $q$ is due to the fact that we do not want to have any $\nabla\rho$ in the right-hand side of Inequality~\eqref{eq:com_est_0}. But now that we obtained the boundedness of $N_2$, we can propagate higher norms of the form $N_{q_2}$ when $q_2 > \frac{6}{1+2a}$. Indeed, \cite[Proposition~6.5]{chong_many-body_2021} also yields for instance
		\begin{equation*}
			\dfrac{1}{\hbar}\Nrm{\com{\nabla V_{\op},\op\,\opm_n} }{\L^\infty} \leq \Nrm{\rho}{L^2}^{1-\theta} \Nrm{\rho}{H^1}^{\theta} \Nrm{\Dhx\!\(\op\,\opm_n\)}{\L^\infty}
		\end{equation*}
		with $\theta = \frac{3}{\fb} - \frac{1}{2} = a + \frac{1}{2} \in(0,1)$, and $\Nrm{\rho}{L^2}$ is controlled by $N_2$. All the $q_2\in(2,\infty)$ follow in the same way finishing the proof of the theorem.
	\end{proof}

{\bf Acknowledgments.} J.C. was supported by the NSF through the RTG grant DMS- RTG 184031. C.S. acknowledges the NCCR SwissMAP and the support of the SNSF through the Eccellenza project PCEFP2\_181153.

\medskip

{\bf Data availability.} All data are contained and available in the paper.


\bibliographystyle{abbrv} 
\bibliography{Vlasov}

\end{document}